\documentclass[amstex,12pt,russian,amssymb]{article}

\usepackage{mathtext}
\usepackage[cp1251]{inputenc}
\usepackage[T2A]{fontenc}
\usepackage[russian]{babel}
\usepackage[dvips]{graphicx}
\usepackage{amsmath}
\usepackage{amssymb}
\usepackage{amsxtra}
\usepackage{latexsym}
\usepackage{ifthen}

\begin{document}

\newcounter{lemma}
\newcommand{\lemma}{\par \refstepcounter{lemma}%
{\bf Лемма \arabic{lemma}.}}

\newcounter{corollary}
\newcommand{\corollary}{\par \refstepcounter{corollary}%
{\bf Следствие \arabic{corollary}.}}

\newcounter{remark}
\newcommand{\remark}{\par \refstepcounter{remark}%
{\bf Замечание \arabic{remark}.}}

\newcounter{theorem}
\newcommand{\theorem}{\par \refstepcounter{theorem}%
{\bf Теорема \arabic{theorem}.}}

\newcounter{proposition}
\newcommand{\proposition}{\par \refstepcounter{proposition}%
{\bf Предложение \arabic{proposition}.}}

\newcommand{\proof}{{\it Доказательство.\,\,}}
\renewcommand{\refname}{\centerline{\bf Список литературы}}

{\bf Е.А.~Севостьянов} (Житомирский государственный университет им.\
И.~Франко)

\medskip
{\bf Є.О.~Севостьянов} (Житомирський державний університет ім.\
І.~Франко)

\medskip
{\bf E.A.~Sevost'yanov} (Zhitomir Ivan Franko State University)

\medskip
{\bf О равностепенной непрерывности классов Орлича--Соболева в
замыкании области}

{\bf Про одностайну неперервність класів Орліча--Соболєва в
замиканні області}

{\bf On equicontinuity of Orlicz--Sobolev class in s closure of a
domain}

\medskip
Изучается поведение гомеоморфизмов классов Орлича--Соболева в
замыкании заданной области. В терминах простых концов регулярных
областей получены теоремы о равностепенной непрерывности указанных
классов. В частности, доказано, что в областях, границы которых
удовлетворяют определённым ограничениям, указанные классы
равностепенно непрерывны, как только их внутренняя дилатация порядка
$p$ имеет мажоранту конечного среднего колебания в каждой точке.

\medskip
Вивчається поведінка гомеоморфізмів класів Орліча--Соболєва в
замиканні заданої області. В термінах простих кінців регулярних
областей отримано теореми про одностайну неперервність вказаних
класів. Зокрема, доведено, що в областях, межі яких задовольняють
визначні обмеження, зазначені класи є одностайно неперервними, як
тільки їх внутряшня в кожній точці.

\medskip
A behavior of homeomorphisms of Orlicz--Sobolev classes in a closure
of a domain is investigated. There are obtained theorems about
equicontinuity of classes mentioned above in terms of prime ends of
regular domains. In particular, it is proved that above classes are
equicontinuous in domains with some restrictions on it's boundaries
provided that the corresponding inner dilatation of order $p$ has a
majorant of finite mean oscillation at every point.

\newpage
{\bf 1. Введение.} Настоящая заметка посвящена изучению локальных
свойств некоторых классов отображений, тесно связанных с классами
Соболева и Орлича--Соболева. Ниже будет показано, что семейство
отображений, аналогичные тем, что рассматриваются в недавней статье
\cite{KR}, являются равностепенно непрерывными в замыкании исходной
области. Подчеркнём, что в первой части статьи речь идёт об
отображениях, удовлетворяющих некоторым ограничениям на внутреннюю
дилатацию порядка $p,$ $p\in (n-1, n].$

Следующие определения могут быть найдены в работе \cite{KR}. Пусть
$\omega$ -- открытое множество в ${\Bbb R}^k$, $k=1,\ldots,n-1$.
Непрерывное отображение $\sigma:\omega\rightarrow{\Bbb R}^n$
называется {\it $k$-мерной поверхностью} в ${\Bbb R}^n$. {\it
Поверхностью} будет называться произвольная $(n-1)$-мерная
поверхность $\sigma$ в ${\Bbb R}^n.$ Поверхность
$\sigma:\omega\rightarrow D$ называется {\it жордановой
поверхностью} в $D$, если $\sigma(z_1)\ne\sigma(z_2)$ при $z_1\ne
z_2$. Далее мы иногда будем использовать $\sigma$ для обозначения
всего образа $\sigma(\omega)\subset {\Bbb R}^n$ при отображении
$\sigma$, $\overline{\sigma}$ вместо $\overline{\sigma(\omega)}$ в
${\Bbb R}^n$ и $\partial\sigma$ вместо
$\overline{\sigma(\omega)}\setminus\sigma(\omega)$. Жорданова
поверхность $\sigma$ в $D$ называется {\it разрезом} области $D$,
если $\sigma$ разделяет $D$, т.\,е. $D\setminus \sigma$ имеет больше
одной компоненты, $\partial\sigma\cap D=\varnothing$ и
$\partial\sigma\cap\partial D\ne\varnothing$.

Последовательность $\sigma_1,\sigma_2,\ldots,\sigma_m,\ldots$
разрезов области $D$ называется {\it цепью}, если:

\medskip

(i) $\overline{\sigma_i}\cap\overline{\sigma_j}=\varnothing$ для
всех $i\ne j$, $i,j= 1,2,\ldots$;

\medskip

(ii) $\sigma_{m-1}$ и $\sigma_{m+1}$ содержатся в различных
компонентах $D\setminus \sigma_m$ для всех $m>1$;

\medskip

(iii) $\cap\,d_m=\varnothing$, где $d_m$ -- компонента $D\setminus
\sigma_m$, содержащая $\sigma_{m+1}$.

\medskip
Согласно определению, цепь разрезов $\{\sigma_m\}$ определяет цепь
областей $d_m\subset D$, таких, что $\partial\,d_m\cap
D\subset\sigma_m$ и $d_1\supset d_2\supset\ldots\supset
d_m\supset\ldots$. Две цепи разрезов $\{\sigma_m\}$ и
$\{\sigma_k^{\,\prime}\}$ называются {\it эквивалентными}, если для
каждого $m=1,2,\ldots$ область $d_m$ содержит все области
$d_k^{\,\prime}$ за исключением конечного числа и для каждого
$k=1,2,\ldots$ область $d_k^{\,\prime}$ также содержит все области
$d_m$ за исключением конечного числа. {\it Конец} области $D$ -- это
класс эквивалентных цепей разрезов $D$.

Пусть $K$ -- конец области $D$ в ${\Bbb R}^n$, $\{\sigma_m\}$ и
$\{\sigma_m^{\,\prime}\}$ -- две цепи в $K$, $d_m$ и
$d_m^{\,\prime}$ -- области, соответствующие $\sigma_m$ и
$\sigma_m^{\,\prime}$. Тогда
$$\bigcap\limits_{m=1}\limits^{\infty}\overline{d_m}\subset
\bigcap\limits_{m=1}\limits^{\infty}\overline{d_m^{\,\prime}}\subset
\bigcap\limits_{m=1}\limits^{\infty}\overline{d_m}\ ,$$ и, таким
образом,
$$\bigcap\limits_{m=1}\limits^{\infty}\overline{d_m}=
\bigcap\limits_{m=1}\limits^{\infty}\overline{d_m^{\,\prime}}\ ,$$
т.\,е. множество
$$I(K)=\bigcap\limits_{m=1}\limits^{\infty}\overline{d_m}$$ зависит
только от $K$ и не зависит от выбора цепи разрезов $\{\sigma_m\}$.
Множество $I(K)$ называется {\it телом конца} $K$.

Число прообразов
$N(y, S)={\rm card}\,S^{-1}(y)={\rm card}\,\{x\in\omega:S(x)=y\},\
y\in{\Bbb R}^n$ будем называть {\it функцией кратности} поверхности
$S.$ Другими словами, $N(y, S)$ -- кратность накрытия точки $y$
поверхностью $S.$ Пусть $\rho:{\Bbb R}^n\rightarrow\overline{{\Bbb
R}^+}$ -- борелевская функция, в таком случае интеграл от функции
$\rho$ по поверхности $S$ определяется равенством:  $\int\limits_S
\rho\,d{\mathcal{A}}:=\int\limits_{{\Bbb R}^n}\rho(y)\,N(y,
S)\,d{\mathcal H}^ky.$
Пусть $\Gamma$ -- семейство $k$-мерных поверхностей $S.$ Борелевскую
функцию $\rho:{\Bbb R}^n\rightarrow\overline{{\Bbb R}^+}$ будем
называть {\it допустимой} для семейства $\Gamma,$ сокр. $\rho\in{\rm
adm}\,\Gamma,$ если
\begin{equation}\label{eq8.2.6}\int\limits_S\rho^k\,d{\mathcal{A}}\geqslant 1\end{equation}
для каждой поверхности $S\in\Gamma.$ Пусть $p\geqslant 1,$ тогда
{\it $p$-модулем} семейства $\Gamma$ назовём величину
$$M_p(\Gamma)=\inf\limits_{\rho\in{\rm adm}\,\Gamma}
\int\limits_{{\Bbb R}^n}\rho^p(x)\,dm(x)\,.$$ Полагаем также
$M(\Gamma):=M_n(\Gamma).$ Далее, как обычно, для множеств $A$, $B$ и
$C$ в ${\Bbb R}^n$, $\Gamma(A,B,C)$ обозначает семейство всех
кривых, соединяющих $A$ и $B$ в $C$.

Следуя \cite{Na}, будем говорить, что конец $K$ является {\it
простым концом}, если $K$ содержит цепь разрезов $\{\sigma_m\}$,
такую, что
$$\lim\limits_{m\rightarrow\infty} M(\Gamma(C, \sigma_m, D))=0
$$
для некоторого континуума $C$ в $D$, где $M$ -- модуль семейства
$\Gamma(C, \sigma_m, D).$

\medskip
Будем говорить, что граница области $D$ в ${\Bbb R}^n$
является {\it локально квазиконформной}, если каждая точка
$x_0\in\partial D$ имеет окрестность $U$, которая может быть
отображена квазиконформным отображением $\varphi$ на единичный шар
${\Bbb B}^n\subset{\Bbb R}^n$ так, что $\varphi(\partial D\cap U)$
является пересечением ${\Bbb B}^n$ с координатной гиперплоскостью.
Говорим, что ограниченная область $D$ в ${\Bbb R}^n$ {\it
регулярна}, если $D$ может быть квазиконформно отображена на область
с локально квазиконформной границей.

\medskip
Как следует из теоремы 4.1 в \cite{Na}, при квазиконформных
отображениях $g$ области $D_0$ с локально квазиконформной границей
на область $D$ в ${\Bbb R}^n$, $n\geqslant2$, существует
естественное взаимно однозначное соответствие между точками
$\partial D_0$ и простыми концами области $D$ и, кроме того,
предельные множества $C(g,b)$, $b\in\partial D_0$, совпадают с телом
$I(P)$ соответствующих простых концов $P$ в $D$.

Если $\overline{D}_P$ является пополнением регулярной области $D$ ее
простыми концами и $g_0$ является квазиконформным отображением
области $D_0$ с локально квазиконформной границей на $D$, то оно
естественным образом определяет в $\overline{D}_p$ метрику
$\rho_0(p_1,p_2)=\left|{\widetilde
{g_0}}^{-1}(p_1)-{\widetilde{g_0}}^{-1}(p_2)\right|$, где
${\widetilde {g_0}}$ продолжение $g_0$ в $\overline {D_0}$,
упомянутое выше.

Если $g_*$ является другим квазиконформным отображением некоторой
области $D_*$ с локально квазиконформной границей на область $D$, то
соответствующая метрика
$\rho_*(p_1,p_2)=\left|{\widetilde{g_*}}^{-1}(p_1)-{\widetilde{g_*}}^{-1}(p_2)\right|$
порождает ту же самую сходимость и, следовательно, ту же самую
топологию в $\overline{D}_P$ как и метрика $\rho_0$, поскольку
$g_0\circ g_*^{-1}$ является квазиконформным отображением между
областями $D_*$ и $D_0$, которое по теореме 4.1 из \cite{Na}
продолжается до гомеоморфизма между $\overline{D_*}$ и
$\overline{D_0}$.

В дальнейшем, будем называть данную топологию в пространстве
$\overline{D}_P$ {\it топологией простых концов} и понимать
непрерывность отображений
$F:\overline{D}_P\rightarrow\overline{D^{\,\prime}}_P$ как раз
относительно этой топологии.

Пусть $\varphi:[0,\infty)\rightarrow[0,\infty)$ -- неубывающая
функция, $f$ -- локально интегрируемая вектор-функция $n$
вещественных переменных $x_1,\ldots,x_n,$ $f=(f_1,\ldots,f_n),$
$f_i\in W_{loc}^{1,1},$ $i=1,\ldots,n.$ Будем говорить, что
$f:D\rightarrow {\Bbb R}^n$ принадлежит классу
$W^{1,\varphi}_{loc},$ пишем $f\in W^{1,\varphi}_{loc},$ если
$\int\limits_{G}\varphi\left(|\nabla f(x)|\right)\,dm(x)<\infty$ для
любой компактной подобласти $G\subset D,$ где $|\nabla
f(x)|=\sqrt{\sum\limits_{i=1}^n\sum\limits_{j=1}^n\left(\frac{\partial
f_i}{\partial x_j}\right)^2}.$ Класс $W^{1,\varphi}_{loc}$
называется классом {\it Орлича--Соболева}.

\medskip
Для отображений класса $W_{loc}^{1,1}$ и произвольного $p\geqslant
1$ корректно определена так называемая {\it внутренняя дилатация
$K_{I, p}(x,f)$ отображения $f$ порядка $p$ в точке $x$},
определяемая равенствами
\begin{equation}\label{eq0.1.1A}
K_{I, p}(x,f)\quad =\quad\left\{
\begin{array}{rr}
\frac{|J(x,f)|}{{l\left(f^{\,\prime}(x)\right)}^p}, & J(x,f)\ne 0,\\
1,  &  f^{\,\prime}(x)=0, \\
\infty, & \text{в\,\,остальных\,\,случаях}
\end{array}
\right.\,.
\end{equation}

Пусть $(X,d)$ и $\left(X^{\,{\prime}},{d}^{\,{\prime}}\right)$ ~---
метрические пространства с расстояниями $d$  и ${d}^{\,{\prime}},$
соответственно. Семейство $\frak{F}$ отображений $f:X\rightarrow
{X}^{\,\prime}$ называется {\it равностепенно непрерывным в точке}
$x_0 \in X,$ если для любого $\varepsilon > 0$ найдётся $\delta
> 0,$ такое, что ${d}^{\,\prime} \left(f(x),f(x_0)\right)<\varepsilon$ для
всех $f \in \frak{F}$ и  для всех $x\in X$ таких, что
$d(x,x_0)<\delta.$ Говорят, что $\frak{F}$ {\it равностепенно
непрерывно}, если $\frak{F}$ равностепенно непрерывно в каждой точке
из $x_0\in X.$ Всюду далее, если не оговорено противное, $d$ -- одна
из метрик в пространстве простых концов относительно области $D,$
упомянутых выше, а $d^{\,\prime}$ -- евклидова метрика. Следуя
\cite[раздел 7.22]{He} будем говорить, что борелева функция
$\rho\colon  X\rightarrow [0, \infty]$ является {\it верхним
градиентом} функции $u\colon X\rightarrow {\Bbb R},$ если для всех
спрямляемых кривых $\gamma,$ соединяющих точки $x$ и $y\in X$
выполняется неравенство $|u(x)-u(y)|\leqslant
\int\limits_{\gamma}\rho\,|dx|.$ Будем также говорить, что в
указанном пространстве $X$ выполняется $(1; p)$-неравенство
Пуанкаре, если найдутся постоянные $C\geqslant 1$ и $\tau>0$ так,
что для каждого шара $B\subset X,$ произвольной ограниченной
непрерывной функции $u\colon X\rightarrow {\Bbb R}$ и любого её
верхнего градиента $\rho$ выполняется следующее неравенство:
$$\frac{1}{\mu(B)}\int\limits_{B}|u-u_B|d\mu(x)\leqslant C\cdot({\rm diam\,}B)\left(\frac{1}{\mu(\tau B)}
\int\limits_{\tau B}\rho^p d\mu(x)\right)^{1/p}\,,$$
где $u_B:=\frac{1}{\mu(B)}\int\limits_{B}u d\mu(x).$ Метрическое
пространство $(X, d, \mu)$ назовём {\it $\widetilde{Q}$-регулярным
по Альфорсу} при некотором $\widetilde{Q}\geqslant 1,$ если при
каждом $x_0\in X,$ некоторой постоянной $C\geqslant 1$ и
произвольного $R<{\rm diam}\,X,$
$\frac{1}{C}R^{\widetilde{Q}}\leqslant \mu(B(x_0, R))\leqslant
CR^{\widetilde{Q}}.$

\medskip
Для областей $D,$ $D^{\,\prime}\subset {\Bbb R}^n,$ $b_0\in D,$
$b_0^{\,\prime}\in D^{\,\prime}$ и произвольной измеримой по Лебегу
функции $Q(x): {\Bbb R}^n\rightarrow [0, \infty],$ такой, что
$Q(x)\equiv 0$ при $x\not\in D,$ обозначим символом
$\frak{F}_{\alpha, b_0, b_0^{\,\prime}, \varphi, Q}(D,
D^{\,\prime})$ семейство всех гомеоморфизмов $f:D\rightarrow
D^{\,\prime}$ класса $W_{loc}^{1, \varphi}$ в $D,$
$f(D)=D^{\,\prime},$ таких что $K_{I, \alpha}(x, f)\leqslant Q(x)$ и
$f(b_0)=b_0^{\,\prime}.$ Справедливо следующее утверждение.

\medskip
\begin{theorem}\label{th7A}{\sl\,
Пусть область $D$ регулярна, область $D^{\,\prime}$ ограничена,
имеет локально квазиконформную границу и, одновременно, является
пространством $n$-регулярным по Альфорсу относительно евклидовой
метрики и меры Лебега в ${\Bbb R}^n,$ в котором выполнено $(1;
\alpha)$-неравенство Пуанкаре, $n-1<\alpha\leqslant n.$ Предположим,
$Q\in L_{loc}^1({\Bbb R}^n),$ $Q(x)=0$ вне области $D,$ заданная
неубывающая функция $\varphi:[0,\infty)\rightarrow[0,\infty)$
удовлетворяет условию (\ref{eqOS3.0a}), и что для каждого $x_0\in
\overline{D}$ выполнено одно из следующих условий:

1) либо $Q\in FMO(\overline{D});$

2) либо в каждой точке $x_0\in \overline{D}$ при некотором
$\varepsilon_0=\varepsilon_0(x_0)>0$ и всех
$0<\varepsilon<\varepsilon_0$
$$
\int\limits_{\varepsilon}^{\varepsilon_0}
\frac{dt}{tq_{x_0}^{\,\frac{1}{n-1}}(t)}<\infty\,,\qquad
\int\limits_{0}^{\varepsilon_0}
\frac{dt}{tq_{x_0}^{\,\frac{1}{n-1}}(t)}=\infty\,,
$$
где
$q_{x_0}(r):=\frac{1}{\omega_{n-1}r^{n-1}}\int\limits_{|x-x_0|=r}Q(x)\,d{\mathcal
H}^{n-1}.$
Тогда каждый элемент $f\in \frak{F}_{\alpha, b_0, b_0^{\,\prime},
\varphi, Q}(D, D^{\,\prime})$ продолжается до непрерывного
отображения $\overline f\colon\overline D_P\rightarrow\overline
{D^{\,\prime}_P}$, при этом, семейство отображений
$\overline{\frak{F}_{b_0, b_0^{\,\prime}, \varphi, Q}(D,
D^{\,\prime})},$ состоящее из всех продолженных таким образом
отображений, является равностепенно непрерывным, а значит, и
нормальным  в $\overline{D}_P$.}
\end{theorem}

\medskip
{\bf 2. Вспомогательные сведения.} Дальнейшее изложение существенно
опираются на аппарат нак называемых нижних $Q$-гомеоморфизмов (см.
\cite[глава~9]{MRSY}). Говорят, что некоторое свойство $P$ выполнено
для {\it почти всех поверхностей} области $D,$ если оно имеет место
для всех поверхностей, лежащих в $D,$ кроме, быть может, некоторого
их подсемейства, модуль которого равен нулю. Будем говорить, что
измеримая по Лебегу функция $\rho:{\Bbb
R}^n\rightarrow\overline{{\Bbb R}^+}$ {\it обобщённо допустима
относительно $p$-модуля} для семейства $\Gamma$ $k$-мерных
поверхностей $S$ в ${\Bbb R}^n,$ сокр. $\rho\in{\rm ext}_p\,{\rm
adm}\,\Gamma,$ если соотношение (\ref{eq8.2.6}) выполнено для почти
всех поверхностей $S$ семейства $\Gamma.$ Следующий класс
отображений представляет собой обобщение квазиконформных отображений
в смысле кольцевого определения по Герингу (\cite{Ge$_3$}) и
отдельно исследуется (см., напр., \cite[глава~9]{MRSY}). Пусть $D$ и
$D^{\,\prime}$ -- заданные области в ${\Bbb R}^n,$ $n\geqslant 2,$
$x_0\in\overline{D}\setminus\{\infty\}$ и $Q:{\Bbb
R}^n\rightarrow(0,\infty)$ -- измеримая по Лебегу функция. Будем
говорить, что $f:D\rightarrow D^{\,\prime}$ -- {\it нижнее
$Q$-отображение в точке $x_0$ относительно $p$-модуля,} как только
\begin{equation}\label{eq1A}
M_p(f(\Sigma_{\varepsilon}))\geqslant \inf\limits_{\rho\in{\rm
ext}_p\,{\rm adm}\,\Sigma_{\varepsilon}}\int\limits_{D\cap A(x_0,
\varepsilon, r_0)}\frac{\rho^p(x)}{Q(x)}\,dm(x)
\end{equation}
для каждого кольца
\begin{equation}\label{eq1B}
A(x_0, \varepsilon, r_0)=\{x\in {\Bbb R}^n:
\varepsilon<|x-x_0|<r_0\}\,,
\end{equation}
$r_0\in(0,d_0),$ $d_0=\sup\limits_{x\in D}|x-x_0|,$
где $\Sigma_{\varepsilon}$ обозначает семейство всех пересечений
сфер $S(x_0, r)$ с областью $D,$ $r\in (\varepsilon, r_0).$ Если
$p=n,$ то будем, что $f$ -- нижнее $Q$-отображение в точке $x_0.$
Будем говорить, что $f$ нижнее $Q$-отображение относительно
$p$-модуля в $A\subset \overline{D},$ если соотношение (\ref{eq1A})
имеет место для каждого $x_0\in A.$

\medskip
Для отображения $f:D\,\rightarrow\,{\Bbb R}^n,$ множества $E\subset
D$ и $y\,\in\,{\Bbb R}^n,$  определим {\it функцию кратности $N(y,
f, E)$} как число прообразов точки $y$ во множестве $E,$ т.е.
\begin{equation}\label{eq1.7A}
N(y, f, E)\,=\,{\rm card}\,\left\{x\in E: f(x)=y\right\}\,,\quad
%
N(f, E)\,=\,\sup\limits_{y\in{\Bbb R}^n}\,N(y, f, E)\,.
\end{equation}

\medskip
Доказательство следующей леммы аналогично доказательству
\cite[теорема~5]{KRSS} и потому опускается.

\medskip
\begin{lemma}{}\label{thOS4.1} {\sl Пусть $D$ -- область в ${\Bbb R}^n,$
$n\geqslant 2,$ $\varphi:(0,\infty)\rightarrow (0,\infty)$ --
неубывающая функция, удовлетворяющая условию
\begin{equation}\label{eqOS3.0a}
\int\limits_{1}^{\infty}\left(\frac{t}{\varphi(t)}\right)^
{\frac{1}{n-2}}dt<\infty\,.
\end{equation}
Если $p>n-1,$ то каждое открытое дискретное отображение
$f:D\rightarrow {\Bbb R}^n$ с конечным искажением класса
$W^{1,\varphi}_{loc}$ такое, что $N(f, D)<\infty,$ является нижним
$Q$-отображением относительно $p$-модуля в каждой точке
$x_0\in\overline{D}$ при
$$Q(x)=N(f, D)\cdot K^{\frac{p-n+1}{n-1}}_{I, \alpha}(x, f),$$
$\alpha:=\frac{p}{p-n+1},$ где внутренняя дилатация $K_{I,\alpha}(x,
f)$ отображения $f$ в точке $x$ порядка $\alpha$ определена
соотношением (\ref{eq0.1.1A}), а кратность $N(f, D)$ определена
вторым соотношением в (\ref{eq1.7A}).}
\end{lemma}

\medskip
Справедливо следующее фундаментальное утверждение, аналог которого
для гомеоморфизмов и частного случая $p=n$ доказан в
\cite[лемма~3]{KR}.

\medskip
\begin{lemma}\label{lem1}
{\sl\, Пусть $n\geqslant 2,$ $p>n-1,$ область $D\subset {\Bbb R}^n$
регулярна, а $D^{\,\prime}\subset {\Bbb R}^n$ ограничена и имеет
локально квазиконформную границу, являющуюся сильно достижимой
относительно $\alpha$-модуля, $\alpha:=\frac{p}{p-n+1}.$ Пусть также
отображение $f:D\rightarrow D^{\,\prime},$ $D^{\,\prime}=f(D),$
является нижним $Q$-отображением в каждой точке $x_0\in \partial D$
относительно $p$-модуля, кроме того, $f$ является открытым,
дискретным и замкнутым. Тогда $f$ продолжается до непрерывного
отображения $f:\overline{D_P}\rightarrow \overline{D_P^{\,\prime}},$
$f(\overline{D_P})=\overline{D_P^{\,\prime}},$ если выполнено
следующее условие. Для каждой точки $x_0\in \partial D$ найдётся
$\varepsilon_0=\varepsilon_0(x_0)>0$ и измеримая по Лебегу функция
$\psi(t):(0, \varepsilon_0)\rightarrow [0,\infty]$ со следующим
свойством: для любого $\varepsilon\in(0, \varepsilon_0)$ выполнено
условие
$$ I(\varepsilon,
\varepsilon_0):=\int\limits_{\varepsilon}^{\varepsilon_0}\psi(t)dt <
\infty\,,\quad I(\varepsilon, \varepsilon_0)\rightarrow
\infty\quad\text{при}\quad\varepsilon\rightarrow 0\,,
$$
и, кроме того, при  $\varepsilon\rightarrow 0$
\begin{equation} \label{eq3.7.2}
\int\limits_{A(x_0, \varepsilon, \varepsilon_0)}
Q(x)\cdot\psi^{\,\alpha}(|x-x_0|)\,dm(x) = o(I^{\alpha}(\varepsilon,
\varepsilon_0))\,,\end{equation}
где, как обычно, сферическое кольцо $A(x_0, \varepsilon,
\varepsilon_0)$ определено (\ref{eq1B}).}
\end{lemma}

\medskip
{\bf 3. О равностепенной непрерывности в областях с $(1;
p)$-неравенством Пуанкаре, регулярных по Альфорсу.} Следующее
утверждение доказано в \cite[предложение~3]{KR} в случае $p=n$ как
для внутренних, так и граничных точек $x_0$ области $D$ (см. также
\cite[лемма~3.8]{KSS}, где изучен случай произвольного $p>n-1$ и
внутренних точек заданной области).  Для граничных точек области $D$
и произвольного $p>n-1$ доказательство данного утверждения
аналогично \cite[предложение~3]{KR}.

\medskip
\begin{lemma}\label{th4}
{\sl\, Пусть $x_0\in  \overline{D},$ $p>n-1,$ и пусть ограниченный
гомеоморфизм $f:D\rightarrow {\Bbb R}^n$ является нижним
$Q$-отображением  в области $D\subset{\Bbb R}^n$ относительно
$p$-модуля, $Q\in L_{loc}^{\frac{n-1}{p-n+1}}({\Bbb R}^n),$ $Q(x)=0$
вне области $D.$ Тогда для каждых $0<r_1<r_2<d_0:=\sup\limits_{x\in
D}|x-x_0|$
\begin{equation*}M_{\alpha}(f(\Gamma(S_1, S_2, D)))\leqslant \int\limits_{A(x_0,
r_1, r_2)}Q^{\frac{n-1}{p-n+1}}(x)
\eta^{\alpha}(|x-x_0|)dm(x)\,,\alpha:=\frac{p}{p-n+1}\,,
\end{equation*}
где $A(x_0, r_1, r_2)=\{x\in {\Bbb R}^n: r_1<|x-x_0|<r_2\},$
$S_i=S(x_0, r_i),$ $i=1,2,$ и $\eta: (r_1, r_2)\rightarrow
[0,\infty]$ -- произвольная измеримая по Лебегу функция такая, что
\begin{equation}\label{eq28*}
\int\limits_{r_1}^{r_2}\eta(r)dr=1\,.
\end{equation}
 }
\end{lemma}
Справедливо следующее утверждение (см.~\cite[предложение~4.7]{AS}).

\medskip
\begin{proposition}\label{pr2}
{\sl Пусть $X$ --- $\beta$-регулярное по Альфорсу метрическое
пространство с мерой, в котором выполняется $(1; p)$-неравенство
Пуанкаре, $\beta-1< p\leqslant \beta.$ Тогда для произвольных
континуумов $E$ и $F,$ содержащихся в шаре $B(x_0, R),$ и некоторой
постоянной $C>0$ выполняется неравенство
$M_p(\Gamma(E, F, X))\geqslant \frac{1}{C}\cdot\frac{\min\{{\rm
diam}\,E, {\rm diam}\,F\}}{R^{1+p-\beta}}.$ }
\end{proposition}

\medskip
Имеет место следующее утверждение, обобщающее
\cite[лемма~3.1]{Sev$_1$} в случае не локально связных границ.

\medskip
\begin{lemma}\label{lem3}
{\sl\, Пусть $p\in [n, n+1/(n-2)),$ $\alpha:=\frac{p}{p-n+1},$
область $D$ регулярна, область $D^{\,\prime}$ ограничена, имеет
локально квазиконформную границу и, одновременно, является
пространством $n$-регулярным по Альфорсу относительно евклидовой
метрики и меры Лебега в ${\Bbb R}^n,$ в котором выполнено $(1;
\alpha)$-неравенство Пуанкаре. Пусть также $P_0$ -- некоторый
простой конец в $E_D,$ а $\sigma_m,$ $m=1,2,\ldots,$ --
соответствующая ему цепь разрезов, лежащих на сферах с центром в
некоторой точке $x_0\in
\partial D$ и радиусов $r_m\rightarrow 0,$ $m\rightarrow\infty.$
Пусть $D_m$ -- соответствующая $P_0$ последовательность
ассоциированных областей, а $C_m$ -- произвольная последовательность
континуумов, принадлежащих $D_m.$

Предположим, $Q\in L_{loc}^{\frac{n-1}{p-n+1}}({\Bbb R}^n),$
$Q(x)=0$ вне области $D,$  $f:D\rightarrow D^{\,\prime}$ -- нижний
$Q$-гомеоморфизм относительно $p$-модуля, $f(D)=D^{\,\prime},$ такой
что $b_0^{\,\prime}=f(b_0)$ для некоторых $b_0\in D$ и
$b_0^{\,\prime}\in D^{\,\prime}.$ Пусть также найдётся
$\varepsilon_0=\varepsilon(x_0)>0,$
такое, что при некотором $0<p^{\,\prime}<\alpha$ выполнено условие
\begin{equation}\label{eq5***}
\int\limits_{A(x_0, \varepsilon, \varepsilon_0)}
Q^{\frac{n-1}{p-n+1}}(x)\cdot\psi^{\,\alpha}(|x-x_0|)\,dm(x)
\leqslant K\cdot I^{p^{\,\prime}}(\varepsilon, \varepsilon_0)\,,
\end{equation}
где сферическое кольцо $A(x_0, \varepsilon, \varepsilon_0)$
определено как в (\ref{eq1B}), а $\psi$ -- некоторая неотрицательная
измеримая функция, такая, что при всех $\varepsilon\in(0,
\varepsilon_0)$
\begin{equation}\label{eq7C}
I(\varepsilon,
\varepsilon_0):=\int\limits_{\varepsilon}^{\varepsilon_0}\psi(t)\,dt
< \infty\,,
\end{equation}
при этом, $I(\varepsilon, \varepsilon_0)\rightarrow \infty$ при
$\varepsilon\rightarrow 0.$

Тогда найдётся число
$\widetilde{\varepsilon_0}=\widetilde{\varepsilon_0}(x_0)\in (0,
\varepsilon_0)$  и $M_0\in {\Bbb N}$ такие, что
$$ {\rm diam\,}f(C_m)\leqslant C\cdot R^{1+\alpha-n}\cdot K\cdot
I^{p^{\,\prime}-\alpha}(r_m, \varepsilon_0)\cdot\Delta(r_m)\,,\quad
m\geqslant M_1\,,
$$
где ${\rm diam\,}f(C_m)$ -- евклидов диаметр множества $f(C_m),$
\begin{equation}\label{eq1.3}
\Delta(\sigma)=\left(
1+\frac{\int\limits_{\widetilde{\varepsilon_0}}^{\varepsilon_0}\psi(t)\,dt}
{\int\limits_{\sigma}^{\widetilde{\varepsilon_0}}\psi(t)\,dt}\right)^{\alpha}\,,\end{equation}
$\delta = \frac{1}{2}\cdot d\left(b_0^{\,\prime},
\partial D^{\,\prime}\right),$ $d$ -- евклидово расстояние
между множествами, $R$ -- радиус шара, содержащего область
$D^{\,\prime},$ а $C$ -- постоянная из предложения \ref{pr2}.}
\end{lemma}

\medskip
\begin{proof} Прежде всего, по определению регулярной области, $D$ может быть отображена
квазиконформно на область $G$ с локально квазиконформной границей
посредством отображения $g:D\rightarrow G.$ Заметим, что $G\ne {\Bbb
R}^n$ ввиду теоремы Лиувилля для квазиконформных отображений (см.
\cite[следствие~2.12, гл.~III]{Ri}). С другой стороны, ввиду
\cite[теорема~3]{KR} $g$ продолжается до гомеоморфизма
$g:\overline{D}_P\rightarrow \overline{G}_P,$ причём ввиду
\cite[теорема~4.1]{Na} существует взаимно однозначное соответствие
между точками границы $G$ и простыми концами в области $G.$ В таком
случае, также существует взаимно однозначное соответствие между
точками границы $G$ и простыми концами в области $D,$ а, значит,
таких простых концов не менее двух.

Пусть теперь $P_1\in E_D$ -- простой конец, не совпадающий с $P_0,$
где $P_0$ -- фиксированный простой конец из условия леммы.
Предположим, $G_m,$ $m=1,2,\ldots,$ -- последовательность областей,
соответствующая простому концу $P_1$ и $x_m\in G$ -- произвольная
последовательность точек, такая что $x_m\rightarrow P_1$ при
$m\rightarrow\infty.$ Можно считать, что $x_m\in G_m$ для всякого
$m\in {\Bbb N}.$ Тогда, так как $f$ имеет непрерывное продолжение на
$\overline{D}_P$ ввиду леммы \ref{lem1}, то $f(x_m)\rightarrow
f(P_1)$ при $m\rightarrow\infty.$ Заметим, что при всех $m\geqslant
m_0$ и некотором $m_0\in {\Bbb N}$
\begin{equation}\label{eq3}
|f(b_0)-f(x_m)|=|b_0^{\,\prime}-f(x_m)|\geqslant
|b_0^{\,\prime}-f(P_1)|-|f(x_m)-f(P_1)|\geqslant \frac{1}{2}\cdot
d(b_0^{\,\prime}, \partial D^{\,\prime}):=\delta\,,
\end{equation}
где $d(b_0^{\,\prime}, \partial D^{\,\prime})$ обозначает евклидово
расстояние между $b_0^{\,\prime}$ и $\partial D^{\,\prime}.$
Построим последовательность континуумов $K_m,$ $m=1,2,\ldots,$
следующим образом. Соединим точку $x_1$ с точкой $b_0$ произвольной
кривой в $D,$ которую мы обозначим через $K_1.$ Далее, соединим
точки $x_2$ и $x_1$ кривой $K_1^{\prime},$ лежащей в $G_1.$
Объединив кривые $K_1$ и $K_1^{\prime},$ получим кривую $K_2,$
соединяющую точки $b_0$ и $x_2.$ И так далее. Пусть на некотором
шаге мы имеем кривую $K_m,$ соединяющую точки $x_m$ и $b_0.$
Соединим точки $x_{m+1}$ и $x_m$ кривой $K_m^{\,\prime},$ лежащей в
$G_m.$ Объединяя между собой кривые $K_m$ и $K_m^{\,\prime},$
получим кривую $K_{m+1}.$ И так далее.

Пусть $C_m,$ $m=1,2,\ldots,$ -- последовательность континуумов в
областях $D_m,$ заданная по условию. Покажем, что найдётся номер
$m_1\in {\Bbb N},$ такой что
\begin{equation}\label{eq4}
D_m\cap K_m=\varnothing\quad\forall\quad m\geqslant m_1\,.
\end{equation}
Предположим, что (\ref{eq4}) не имеет места, тогда найдутся
возрастающая последовательность номеров $m_k\rightarrow\infty,$
$k\rightarrow\infty,$ и последовательность точек $\xi_k\in
K_{m_k}\cap D_{m_k},$ $m=1,2,\ldots,\,.$ Тогда, с одной стороны,
$\xi_k \rightarrow P_0$ при $k\rightarrow\infty.$

\medskip
Рассмотрим следующую процедуру. Заметим, что возможны два случая:
либо все элементы $\xi_k$ при $k=1,2,\ldots$ принадлежат множеству
$D\setminus G_1,$ либо найдётся номер $k_1$ такой, что $\xi_{k_1}\in
G_1.$ Далее, рассмотрим последовательность $\xi_k,$ $k>k_1.$
Заметим, что возможны два случая: либо $\xi_k$ при $k>k_1$
принадлежат множеству $D\setminus G_2,$ либо найдётся номер
$k_2>k_1$ такой, что $\xi_{k_2}\in G_2.$ И так далее. Предположим,
элемент $\xi_{k_{l-1}}\in G_{l-1}$ построен. Заметим, что возможны
два случая: либо $\xi_k$ при $k>k_{l-1}$ принадлежат множеству
$D\setminus G_l,$ либо найдётся номер $k_l>k_{l-1}$ такой, что
$\xi_{k_l}\in G_l.$ И так далее. Эта процедура может быть как
конечной (оборваться на каком-то $l\in {\Bbb N}$), так и
бесконечной, в зависимости от чего мы имеем две ситуации:

1) либо найдутся номера $n_0\in {\Bbb N}$ и $l_0\in {\Bbb N}$ такие,
что $\xi_k\in D\setminus G_{n_0}$ при всех $k>l_0;$

2) либо для каждого $l\in {\Bbb N}$ найдётся элемент $\xi_{k_l}$
такой, что $\xi_{k_l}\in G_l,$ причём последовательность $k_l$
является возрастающей по $l\in {\Bbb N}.$

\medskip
Рассмотрим каждый из этих случаев и покажем, что в обоих из них мы
приходим к противоречию. Пусть имеет место ситуация 1), тогда
заметим, что все элементы последовательности $\xi_k$ принадлежат
$K_{n_0},$ откуда вытекает существование подпоследовательности
$\xi_{k_r},$ $r=1,2,\ldots,$ сходящейся при $r\rightarrow\infty$ к
некоторой точке $\xi_0\in D.$ Однако, с другой стороны $\xi_k\in
D_{m_k}$ и, значит, $\xi_0\in \bigcap\limits_{m=1}^{\infty}
\overline{D_m}\subset
\partial D$ (см. \cite[предложение~1]{KR} по этому поводу).
Полученное противоречие говорит о том, что случай 1) невозможен.
Пусть имеет место случай 2), тогда одновременно $\xi_k\rightarrow
P_0$ и $\xi_k\rightarrow P_1$ при $k\rightarrow\infty.$ В силу
непрерывного продолжения $f$ на $\overline{D}_P$ отсюда имеем, что
$f(\xi_k)\rightarrow f(P_0)$ и $f(\xi_k)\rightarrow f(P_1)$ при
$k\rightarrow\infty,$ откуда $f(P_0)=f(P_1),$ что противоречит
гомеоморфности продолжения $f$ в $\overline{D}_P.$ Полученное
противоречие указывает на справедливость соотношения (\ref{eq4}).

Положим теперь $\widetilde{\varepsilon_0}=\min\{\varepsilon_0,
r_{m_1+1}\},$ и пусть $M_0$ -- натуральное число, такое что
$r_m<\widetilde{\varepsilon_0}$ при всех $m\geqslant M_0.$
Рассмотрим измеримую функцию
$$\eta_{m}(t)= \left\{
\begin{array}{rr}
\psi(t)/I(r_m, \widetilde{\varepsilon_0}), &   t\in (r_m, \widetilde{\varepsilon_0}),\\
0,  &  t\not\in (r_m, \widetilde{\varepsilon_0})\,,
\end{array}
\right.$$
%
%
где, как и прежде, величина $I(a, b)$ определяется соотношением
$I(a, b)=\int\limits_a^b\psi(t)\,dt.$ Заметим, что функция
$\eta_m(t)$ удовлетворяет соотношению вида (\ref{eq28*}), где вместо
$r_1$ и $r_2$ участвуют $r_m$ и $\widetilde{\varepsilon_0},$
соответственно. Заметим, что ввиду соотношения (\ref{eq4}), а также
по определению разрезов $\sigma_m\subset r_m,$ $\Gamma\left(C_m,
K_m, D\right)>\Gamma(S(x_0, r_m), S(x_0, \widetilde{\varepsilon_0}),
D)$ и значит, $f(\Gamma\left(C_m, K_m, D\right))>f(\Gamma(S(x_0,
r_m), S(x_0, \widetilde{\varepsilon_0}),D)),$ откуда
$$M_{\alpha}(f(\Gamma\left(C_m, K_m, D\right)))\leqslant M_{\alpha}(f(\Gamma(S(x_0,
r_m), S(x_0, \widetilde{\varepsilon_0}), D))$$
(см. \cite[теорема~6.4]{Va}). В таком случае, согласно лемме
\ref{th4}, мы получим, что
$$M_{\alpha}\left(\Gamma\left(f(C_m), f(K_m),
D^{\,\prime}\right)\right)=$$
\begin{equation}\label{eq37***}
=M_{\alpha}\left(f\left(\Gamma\left(C_m, K_m,
D\right)\right)\right)\leqslant M_{\alpha}(f(\Gamma(S(x_0, r_m),
S(x_0, \widetilde{\varepsilon_0}), D )))\leqslant
\end{equation}
$$\leqslant\,\frac{K\cdot I^{p^{\,\prime}}(r_m, \varepsilon_0)}{I^{\alpha}(r_m,
\widetilde{\varepsilon_0})}=K\cdot I^{p^{\,\prime}-\alpha}(r_m,
\varepsilon_0)\cdot\Delta(r_m)\,,\quad m\geqslant M_0\,,$$
где $\Delta(r_m)$ определяется из соотношения (\ref{eq1.3}) при
$\sigma=r_m.$
С другой стороны, по предложению \ref{pr2}
\begin{equation}\label{eq11}
\min\{{\rm diam\,}f(C_m), {\rm diam\,}f(K_m)\}\leqslant
CR^{1+\alpha-n}M_{\alpha}\left(\Gamma\left(f(C_m), f(K_m),
D^{\,\prime}\right)\right)\,,
\end{equation}
где $R$ -- радиус шара, содержащего область $D^{\,\prime}.$

Поскольку из (\ref{eq37***}) вытекает, что
$M_{\alpha}\left(\Gamma\left(f(C_m), f(K_m),
D^{\,\prime}\right)\right)\rightarrow 0$ при $m\rightarrow\infty,$ а
${\rm diam\,}f(K_m)\geqslant \delta$ ввиду (\ref{eq3}), то из
(\ref{eq11}) следует, что при некотором $M_1\geqslant M_0$ и всех
$m\geqslant M_1$
\begin{equation}\label{eq12B}
{\rm diam\,}f(C_m)\leqslant C\cdot R^{1+\alpha-n}\cdot
M_{\alpha}\left(\Gamma\left(f(C_m), f(K_m),
D^{\,\prime}\right)\right)\,.
\end{equation}
Тогда из (\ref{eq37***}) и (\ref{eq12B}) вытекает, что
$$
{\rm diam\,}f(C_m)\leqslant C\cdot R^{1+\alpha-n}\cdot K\cdot
I^{p^{\,\prime}-\alpha}(r_m, \varepsilon_0)\cdot\Delta(r_m)\,,\quad
m\geqslant M_1\,.
$$
Лемма доказана.~$\Box$
\end{proof}

\medskip
Для заданных областей $D,$ $D^{\,\prime}\subset {\Bbb R}^n,$ $n\ge
2,$ $n-1<p\leqslant n,$ измеримой по Лебегу функции $Q(x):{\Bbb
R}^n\rightarrow [0, \infty],$ $Q(x)\equiv 0$ при $x\not\in D,$
$b_0\in D,$ $b_0^{\,\prime}\in D^{\,\prime},$ обозначим через
$\frak{G}_{p, b_0, b_0^{\,\prime}, Q}\left(D, D^{\,\prime}\right)$
семейство всех нижних $Q$-гомеоморфизмов $f:D\rightarrow
D^{\,\prime}$ относительно $p$-модуля, таких что
$f(D)=D^{\,\prime},$ $b_0^{\,\prime}=f(b_0).$ В наиболее общей
ситуации основное утверждение настоящего раздела может быть
сформулировано следующим образом.

\medskip
\begin{lemma}\label{lem3A}
{\sl\, Пусть $p\in [n, n+1/(n-2)),$ $\alpha:=\frac{p}{p-n+1},$
область $D$ регулярна, область $D^{\,\prime}$ ограничена, имеет
локально квазиконформную границу и, одновременно, является
пространством $n$-регулярным по Альфорсу относительно евклидовой
метрики и меры Лебега в ${\Bbb R}^n,$ в котором выполнено $(1;
\alpha)$-неравенство Пуанкаре.

Предположим, что $Q\in L_{loc}^{\frac{n-1}{p-n+1}}({\Bbb R}^n),$ и
что для каждого $x_0\in \overline{D}$ найдётся
$\varepsilon_0=\varepsilon(x_0)>0,$
такое, что при некотором $0<p^{\,\prime}<\alpha$ выполнено условие
(\ref{eq5***}), где сферическое кольцо $A(x_0, \varepsilon,
\varepsilon_0)$ определено как в (\ref{eq1B}), а $\psi$ -- некоторая
неотрицательная измеримая функция, такая, что при всех
$\varepsilon\in(0, \varepsilon_0)$ выполнено условие (\ref{eq7C}),
при этом, $I(\varepsilon, \varepsilon_0)\rightarrow \infty$ при
$\varepsilon\rightarrow 0.$

Тогда каждое $f\in\frak{G}_{p, b_0, b_0^{\,\prime}, Q}\left(D,
D^{\,\prime}\right)$ продолжается до гомеоморфизма
$f:\overline{D}_P\rightarrow \overline{D^{\,\prime}}_P,$ при этом
семейство таким образом продолженных отображений является
равностепенно непрерывным в $\overline{D}_P.$ }
\end{lemma}

\begin{proof} Каждое отображение $f$ имеет непрерывное продолжение на
$\overline{D}_P$ в силу леммы \ref{lem1}. Равностепенная
непрерывность семейства $\frak{G}_{p, b_0, b_0^{\,\prime},
Q}\left(D, D^{\,\prime}\right)$ во внутренних точках области $D$
следует, например, из комбинации леммы \ref{th4} и
\cite[лемма~2.4]{GSS}.

Осталось показать равностепенную непрерывность  $\frak{G}_{p, b_0,
b_0^{\,\prime}, Q}\left(D, D^{\,\prime}\right)$ на $E_D.$

Предположим противное, а именно, что семейство отображений
$\frak{G}_{p, b_0, b_0^{\,\prime}, Q}\left(D, D^{\,\prime}\right)$
не является равностепенно непрерывным в некоторой точке $P_0\in
E_D.$ Тогда найдутся число $a>0,$ последовательность $P_k\in
\overline{D}_P,$ $k=1,2,\ldots$ и элементы $f_k\in\frak{G}_{b_0,
b_0^{\,\prime}, Q}\left(D, D^{\,\prime}\right)$ такие, что $d(P_k,
P_0)<1/k$ и
\begin{equation}\label{eq6}
|f_k(P_k)-f_k(P_0)|\geqslant a\quad\forall\quad k=1,2,\ldots,\,.
\end{equation}
Ввиду возможности непрерывного продолжения каждого $f_k$ на границу
$D$ в терминах простых концов, для всякого $k\in {\Bbb N}$ найдётся
элемент $x_k\in D$ такой, что $d(x_k, P_k)<1/k$ и
$|f_k(x_k)-f_k(P_k)|<1/k.$ Тогда из (\ref{eq6}) вытекает, что
\begin{equation}\label{eq7D}
|f_k(x_k)-f_k(P_0)|\geqslant a/2\quad\forall\quad k=1,2,\ldots,\,.
\end{equation}
Аналогично, в силу непрерывного продолжения отображения $f_k$ в
$\overline{D_P}$ найдётся последовательность $x_k^{\,\prime}\in D,$
$x_k^{\,\prime}\rightarrow P_0$ при $k\rightarrow \infty$ такая, что
$|f_k(x_k^{\,\prime})-f_k(P_0)|<1/k$ при $k=1,2,\ldots\,.$ Тогда из
(\ref{eq7D}) вытекает, что
\begin{equation}\label{eq8B}
|f_k(x_k)-f_k(x_k^{\,\prime})|\geqslant a/4\quad\forall\quad
k=1,2,\ldots\,,\,.
\end{equation}
Пусть $\sigma_m,$ $m=1,2,\ldots,$ -- соответствующая $P_0$ цепь
разрезов, лежащих на сферах с центром в некоторой точке $x_0\in
\partial D$ и радиусов $r_m\rightarrow 0,$ $m\rightarrow\infty.$
Пусть $D_m$ -- соответствующая $P_0$ последовательность
ассоциированных областей. Не ограничивая общности рассуждений, можно
считать, что $x_k$ и $x_k^{\,\prime}$ принадлежат области $D_k.$
Соединим точки $x_k$ и $x_k^{\,\prime}$ кривой $C_k$ лежащей в
$D_k.$ Тогда по лемме \ref{lem3} мы получим, что ${\rm
diam}\,f(C_k)\rightarrow 0$ при $k\rightarrow \infty,$ что
противоречит неравенству (\ref{eq8B}). Полученное противоречие
указывает на то, что исходное предположение об отсутствии
равностепенной непрерывности семейства $\frak{G}_{p, b_0,
b_0^{\,\prime}, Q}\left(D, D^{\,\prime}\right)$ было
неверным.~$\Box$
\end{proof}

\medskip
Из леммы \ref{lem3A} на основе рассуждений из \cite[лемма~3.1 и
детали доказательства теоремы~4.2]{GSS$_1$} (см. также
\cite[лемма~2.3.1]{Sev$_7$}), получаем следующее утверждение.

\medskip
\begin{theorem}\label{th8}{\sl\,
Пусть $p\in [n, n+1/(n-2)),$ $\alpha:=\frac{p}{p-n+1},$ область $D$
регулярна, область $D^{\,\prime}$ ограничена, имеет локально
квазиконформную границу и, одновременно, является пространством
$n$-регулярным по Альфорсу относительно евклидовой метрики и меры
Лебега в ${\Bbb R}^n,$ в котором выполнено $(1; \alpha)$-неравенство
Пуанкаре.

Предположим, что $Q\in L_{loc}^{\frac{n-1}{p-n+1}}({\Bbb R}^n)$ и
выполнено одно из следующих условий:

1) либо в каждой точке $x_0\in \overline{D}$ при некотором
$\varepsilon_0=\varepsilon_0(x_0)>0$ и всех
$0<\varepsilon<\varepsilon_0$
$$
\int\limits_{\varepsilon}^{\varepsilon_0}
\frac{dt}{t^{\frac{n-1}{\alpha-1}}\widetilde{q}_{x_0}^{\,\frac{1}{\alpha-1}}(t)}<\infty\,,\qquad
\int\limits_{0}^{\varepsilon_0}
\frac{dt}{t^{\frac{n-1}{\alpha-1}}\widetilde{q}_{x_0}^{\,\frac{1}{\alpha-1}}(t)}=\infty\,,
$$
где
$\widetilde{q}_{x_0}(r):=\frac{1}{\omega_{n-1}r^{n-1}}\int\limits_{|x-x_0|=r}Q^{\frac{n-1}{p-n+1}}(x)\,d{\mathcal
H}^{n-1};$

2) либо $Q^{\frac{n-1}{p-n+1}}\in FMO(\overline{D}).$
Тогда каждое $f\in\frak{G}_{p, b_0, b_0^{\,\prime}, Q}\left(D,
D^{\,\prime}\right)$ продолжается до гомеоморфизма
$f:\overline{D}_P\rightarrow \overline{D^{\,\prime}}_P,$ при этом
семейство таким образом продолженных отображений является
равностепенно непрерывным в $\overline{D}_P.$}
\end{theorem}

\medskip
\begin{proof}
В силу леммы \ref{thOS4.1} каждое отображение $\frak{F}_{\alpha,
b_0, b_0^{\,\prime}, \varphi, Q}(D, D^{\,\prime})$ является нижним
$B$-отображением относительно $p$-модуля при $B(x)=
Q^{\frac{p-n+1}{n-1}}(x, f),$ где $p$ находится из условия
$\alpha=\frac{p}{p-n+1}.$ Однако, относительно $B(x)$ выполнены
условия 1) и 2) теоремы \ref{th8}, поскольку
$B^{\frac{n-1}{p-n+1}}(x)=Q(x),$ где $Q$ удовлетворяет соотношениям
1)-2) теоремы \ref{th7A}. Оставшаяся часть утверждения вытекает из
теоремы \ref{th8}.~$\Box$
\end{proof}

\medskip
КОНТАКТНАЯ ИНФОРМАЦИЯ

\medskip
\noindent{{\bf Евгений Александрович Севостьянов} \\
Житомирский государственный университет им.\ И.~Франко\\
ул. Большая Бердичевская, 40 \\
г.~Житомир, Украина, 10 008 \\ тел. +38 066 959 50 34 (моб.),
e-mail: esevostyanov2009@mail.ru}

\end{document}